\documentclass{amsart}
\usepackage{multirow}
\usepackage{multicol}
\usepackage{amsfonts}
\usepackage{comment}
\usepackage{amsmath}
\usepackage{amssymb}
\usepackage{amsthm}
\usepackage{color}
\usepackage{url}

\newcommand{\mA}{\mathcal{A}}

\newcommand{\re}{\mathbb{R}}

\newcommand{\N}{\mathbb{N}}

\newcommand{\lmd}{\lambda}

\newcommand{\Dt}{\Delta}
\def\af{\alpha}
\def\bt{\beta}

\newcommand{\mt}[1]{\mathtt{#1}}

\newcommand{\reff}[1]{(\ref{#1})}

\newcommand{\mc}[1]{\mathcal{#1}}

\definecolor{red}{rgb}{1,0,0}

\newcommand{\bdes}{\begin{description}}
\newcommand{\edes}{\end{description}}

\newcommand{\bal}{\begin{align}}
\newcommand{\eal}{\end{align}}

\newcommand{\bnum}{\begin{enumerate}}
\newcommand{\enum}{\end{enumerate}}

\newcommand{\bit}{\begin{itemize}}
\newcommand{\eit}{\end{itemize}}

\newcommand{\bea}{\begin{eqnarray}}
\newcommand{\eea}{\end{eqnarray}}
\newcommand{\be}{\begin{equation}}
\newcommand{\ee}{\end{equation}}

\newcommand{\baray}{\begin{array}}
\newcommand{\earay}{\end{array}}

\newcommand{\bsry}{\begin{subarray}}
\newcommand{\esry}{\end{subarray}}

\newcommand{\bca}{\begin{cases}}
\newcommand{\eca}{\end{cases}}

\newcommand{\bcen}{\begin{center}}
\newcommand{\ecen}{\end{center}}

\newcommand{\bbm}{\begin{bmatrix}}
\newcommand{\ebm}{\end{bmatrix}}

\newcommand{\bmx}{\begin{matrix}}
\newcommand{\emx}{\end{matrix}}

\newcommand{\bpm}{\begin{pmatrix}}
\newcommand{\epm}{\end{pmatrix}}

\newcommand{\btab}{\begin{tabular}}
\newcommand{\etab}{\end{tabular}}

\newtheorem{theorem}{Theorem}[section]

\theoremstyle{definition}
\newtheorem{example}[theorem]{Example}

\newtheorem{alg}[theorem]{Algorithm}
\newtheorem{remark}[theorem]{Remark}

\numberwithin{equation}{section}

\begin{document}

\title[Detecting Copositive Matrices and Tensors]
{A Complete Semidefinite Algorithm for Detecting Copositive Matrices and Tensors}

\author{Jiawang Nie}
\address{Department of Mathematics,
University of California San Diego,
9500 Gilman Drive, La Jolla, CA, USA, 92093.}
\email{njw@math.ucsd.edu}

\author{Zi Yang}
\email{ziy109@ucsd.edu}

\author{Xinzhen Zhang}
\address{School of Mathematics, Tianjin University,
Tianjin 300072, China.}
\email{xzzhang@tju.edu.cn}

\subjclass[2010]{15A69,15B48,90C22}

\date{}

\keywords{copositivity, matrix, tensor,
semidefinite relaxation}

\begin{abstract}
A real symmetric matrix (resp., tensor) is said to be
copositive if the associated quadratic (resp., homogeneous) form
is greater than or equal to zero over the nonnegative orthant.
The problem of detecting their copositivity is NP-hard.
This paper proposes a complete semidefinite relaxation algorithm
for detecting the copositivity of a matrix or tensor.
If it is copositive, the algorithm can get a certificate for the copositivity.
If it is not, the algorithm can get a point that refutes the copositivity.
We show that the detection can be done by
solving a finite number of semidefinite relaxations,
for all matrices and tensors.
\end{abstract}

\maketitle

\section{Introduction}
\label{sec:introdution}

\subsection{Copositive matrices}

A real symmetric matrix $A \in \re^{n \times n}$
is said to be {\it copositive} if
\[
x^TA x \geq 0 \quad \forall \, x \in \re_+^n,
\]
where $\re_+^n$ is the nonnegative orthant
(i.e., the set of nonnegative vectors).
If $x^TA x > 0$ for all $0 \ne x \in \re_+^n$,
then $A$ is said to be {\it strictly copositive}.
The set of all $n\times n$ copositive matrices
is a cone in $\re^{n \times n}$,
which is denoted as $\mc{COP}_n$.
Copositive matrices were introduced in \cite{Motzkin1818Copositive}.
They have broad applications, e.g.,
in quadratic programming \cite{Burer2009copositive},
dynamical systems and control theory \cite{Jacobson2000Extensions,MasSho07},
graph theory \cite{DeKlerk2002Approximation,Dukanovic2010Copositive},
complementarity problems \cite{Facchinei2007Finite}.
We refer to \cite{Bomz12,Dur10}
for surveys on copositive optimization.

A basic problem in optimization
is the detection of copositive matrices.
Let $\mc{S}_+^n$ be the cone of $n\times n$
real symmetric positive semidefinite (psd) matrices,
and $\mc{N}_+^n$ be the cone of
$n\times n$ real symmetric matrices whose entries are all nonnegative.
Clearly, it holds that
\be
\mc{S}_+^n + \mc{N}_+^n  \, \subseteq \, \mc{COP}_n.
\ee
For $n \leq 4$, the above inclusion is an equality;
for $n \geq 5$, the equality does not hold any more \cite{Dia62}.
For instance, the Horn matrix \cite{HalNew63}
is copositive, but it is not a sum of
psd and nonnegative matrices.
Checking membership of the cone $\mc{COP}_n$
is NP-hard~\cite{DicGij14}.
As shown in \cite{Kaplan00}, a matrix $A$ is copositive
if and only if it does not have a principal submatrix
that has a negative eigenvalue with a positive eigenvector.
To apply this testing, one needs to check eigenvalues for all principal submatrices,
which grows exponentially in the dimension.
For the case $n=5$, when the diagonal entries
are all ones, $A$ is copositive if and only if
the polynomial $\|x\|^2(\sum_{i,j=1}^5 A_{ij} x_i^2x_j^2)$
is a sum of squares \cite{DDGH13}.
When off-diagonal entries are nonpositive,
$A$ is copositive if and only if $A$ is positive semidefinite \cite{HirSee10}.
When a matrix is tridiagonal or acyclic,
its copositivity can be detected in linear time
\cite{Bom00,ikramov2002linear}.
For testing copositivity for general matrices,
there exist methods based on simplicial partition. We refer to
\cite{Bundfuss2008Algorithmic,sponsel2012improved}
and the references  therein.
Another approach for testing copositivity is to use
the difference of convexity \cite{Bomze2013Copositivity,DurHU13}.
A survey about existing results and open problems
for copositive matrices can be found in \cite{BerDurSha15}.

\subsection{Copositive tensors}

Matrices can be viewed as tensors of order $2$.
The concept of copositivity can be naturally generalized to tensors,
as in Qi~\cite{Qi2013Symmetric}.
A tensor $\mA$ is a multi-dimensional array
\[
\mA \, := \, (\mA_{i_1 \ldots i_m}),
\]
with indices $1 \leq i_1, \ldots, i_m \leq n$.
The number $m$ is the order.
Such $\mA$ is called an $n$-dimensional tensor of order $m$.
In some applications, we often have symmetric tensors.
The tensor $\mathcal{A}$ is {\it symmetric} if
$\mA_{i_1 i_2 \ldots i_m} = \mA_{j_1 j_2 \ldots j_m}$
whenever $(i_1,i_2,\ldots,i_m)$ is a permutation of $(j_1,j_2,\ldots,j_m)$.
We denote by $\mt{S}^m(\re^n)$ the space of
symmetric tensors of order $m$ over the vector space $\re^n$.
For $\mA \in \mt{S}^m(\re^n)$, define the homogeneous polynomial
\be \label{df:A(x)}
\mA (x) \, := \, \sum_{1\le i_1,i_2,\cdots,i_m \le n}
\mA_{i_1 i_2 \cdots i_m} x_{i_1} x_{i_2} \cdots x_{i_m}.
\ee
Clearly, $\mA(x)$ is a homogeneous polynomial
(i.e., a form) of degree $m$ in the variable $x:=(x_1,\ldots, x_n)$.
If $\mA(x) \geq 0$ for all $x \in \re^n$,
$\mA$ is said to be {\it positive semidefinite} (psd).
If $\mA(x) \geq 0$ for all $x \in \re_+^n$,
$\mA$ is said to be {\it copositive}.
Similarly, if $\mA(x) > 0$ for all $0 \ne x \in \re_+^n$,
$\mA$ is said to be {\it strictly} copositive.
Denote by $\mc{COP}_{m,n}$ the cone of all copositive tensors in $\mt{S}^m(\re^n)$.
Clearly, when the order $m=2$, positive semidefinite
(resp., copositive) tensors are the same as
positive semidefinite (resp., copositive) matrices.
Like the matrix case, copositive tensors also have wide applications,
e.g., in complementary problems
\cite{Che2016Positive,song2014properties,Song2016Tensor},
physics and hypergraphs \cite{Chen2016Copositive,kannike2016vacuum}.
We refer to \cite{CHQ17,Qi2013Symmetric,Song2014Necessary,song2014properties}
for more applications and properties of copositive tensors.

Detecting copositivity of symmetric tensors is also
a mathematically challenging question.
This problem is also NP-hard, since it includes
testing matrix copositivity as a special case.
If the off-diagonal entries of a symmetric tensor $\mA$ are nonpositive,
then $\mA$ is copositive if and only if $\mA$
is positive semidefinite \cite{Qi2013Symmetric}.
There also exists a characterization of copositive tensors
by the eigenpairs of its principal subtensors \cite{Song2014Necessary}.
Like the matrix case, the copositivity of a tensor
can be tested by algorithms based on simplicial partition.
We refer to \cite{Chen2016Copositive,CHQ17}
and the references therein.

\subsection{Contributions}

In the prior existing methods for detecting copositivity,
most of them can detect the copositivity
if the matrix or tensor lies in the interior of the copsitive cone
$\mc{COP}_n$ or $\mc{COP}_{m,n}$,
or if it lies outside the copositive cone.
If the matrix or tensor lies on the boundary,
then these methods typically have difficulty
in detecting the copositivity.
If it is close to the boundary,
they are often very expensive for doing the detection.

In this paper, we propose a new algorithm for detecting copositivity.
It is based on Lasserre type semidefinite relaxations
and optimality conditions of polynomial optimization.
To be precisely, we construct a hierarchy of semidefinite
relaxations for checking copositivity.
The construction uses semidefinite relaxation techniques
that are developed in the recent work~\cite{Nie2017Tight}.
If a tensor/matrix $\mA$ is copositive,
we can get a certificate for the copositivity.
If it is not copositive, we can compute a point
$u \in \re_+^n$ such that $\mA(u) <0$.
Such a point $u$ refutes the copositivity of $\mA$.
No matter a matrix is copositive or not,
the testing of copositivity can be done by the algorithm
in finitely many steps.
Even if the matrix or tensor lies on the boundary of the copositive cone,
the algorithm also terminates in finitely many steps.
In other words, for every matrix and tensor,
its copositivity can be detected
by solving a finite number of semidefinite relaxations.
This is why we call it a {\it complete semidefinte}
algorithm for detecting copositivity.
To the best of the authors' knowledge,
this is the first semidefinite relaxation algorithm that can detect copositivity
and that can terminate in finitely many steps
for all matrices and tensors.

The paper is organized as follows.
Section~2 reviews some preliminaries in polynomial optimization.
Section~3 gives the complete semidefinite algorithm.
Section~4 presents numerical experiments of the algorithm.
Section~5 draws some conclusions and makes discussions.

\section{Preliminaries}

The symbol $\mathbb N$ stands
for the set of nonnegative integers, and $\mathbb R$
for the real field.
For $x :=(x_1, \ldots, x_n) \in {\mathbb R}^n$
and $\af := (\af_1, \ldots, \af_n) \in \N^n$, denote
\[
x^\alpha := x_1^{\alpha_1}\cdots x_n^{\alpha_n}, \quad
|\alpha|:=\alpha_1+\cdots+\alpha_n.
\]
For an integer $m >0$, denote the set
\[
{\mathbb{N}}_m^n \, := \,
\{\alpha\in {\mathbb{N}}^n|\ |\alpha|\le m \}.
\]
The symbol ${\mathbb R}[x]$ denotes the ring of polynomials in
$x$ with real coefficients, and ${\mathbb R}[x]_k$
denotes the space of polynomials in $\re[x]$
with degrees at most $k$.
For a symmetric matrix $X$, the inequality $X\succeq 0$ means
$X$ is positive semidefinite.
The superscript $^T$ denotes the transpose of a matrix or vector.
We use $[x]_m$ to denote the column vector of
all monomials in $x$ and of degrees at most $m$
(they are ordered in the graded lexicographical ordering), i.e.,
\[
[x]_m  := [1, \, x_1, \ldots, x_n, \, x_1^2,\,  x_1x_2,\,
\ldots,\, x_{n-1}x_n^{d-1}, x_n^m \,]^T.
\]
For a vector $x$, $\| x \|$ denotes its Euclidean norm.
In the space $\re^n$, $e$ denotes the vector of all ones,
while $e_i$ denotes the $i$th unit vector in the canonical basis.
For a real number $t$,
$\lceil t \rceil$ denotes the smallest integer not smaller than $t$.

The set $\re^{ \N^n_d }$ is the space of all real vectors
that are labeled by $\af \in \N_d^n$. That is, every
$y \in \re^{ \N^n_d }$ can be labeled as
\[
y \, = \, (y_\af)_{ \af \in \N_d^n }.
\]
Such $y$ is called a
{\it truncated multi-sequence} (tms) of degree $d$ \cite{ATKMP}.
For a polynomial $f \in \re[x]_r$ that is written as
\[
f = \sum_{ |\af| \leq \N^n_r } f_\af x^\af,
\]
with $r \leq d$, we define the operation
\be \label{<f,y>}
\langle f, y \rangle = \sum_{ |\af| \leq \N^n_r } f_\af y_\af.
\ee
Note that $\langle f, y \rangle$ is linear in $y$ for fixed $f$,
and is linear in $f$ for fixed $y$.
For a polynomial $q \in \re[x]_{2k}$ and the integer
$t = \lceil k- \deg(q)/2 \rceil$, the outer product
$q(x)[x]_t[x]_t^T$
is a symmetric matrix of length $\binom{n+t}{t}$.
It can be expanded as
\[
q(x)[x]_t[x]_t^T \, = \, \sum_{ \af \in \N_{2k}^n }
x_\af  Q_\af,
\]
for constant symmetric matrices $Q_\af$.
For $y \in \re^{ \N^n_{2k} }$, denote the symmetric matrix
\be \label{df:Lf[y]}
L_{q}^{(k)}[y] \, := \, \sum_{ \af \in \N_{2k}^n }
y_\af  Q_\af.
\ee
It is called the $k$th {\it localizing matrix} of $q$ and generated by $y$.
For given $q$, $L_{q}^{(k)}[y]$ is linear in $y$.
Clearly, if $q(u) \geq 0$ and $y = [u]_{2k}$, then
\[ L_{q}^{(k)}[y] = q(u) [u]_t[u]_t^T \succeq 0. \]
For instance, if $n=k=2$
and $q= 1 - x_1-x_1x_2$, then
\[
L_q^{(2)}[y]=\left [\begin{matrix}
y_{00}-y_{10}-y_{11} &  y_{10}-y_{20}-y_{21} &  y_{01}-y_{11}-y_{12} \\
y_{10}-y_{20}-y_{21} &  y_{20}-y_{30}-y_{31} &  y_{11}-y_{21}-y_{22} \\
y_{01}-y_{11}-y_{12} &  y_{11}-y_{21}-y_{22} &  y_{02}-y_{12}-y_{13} \\
\end{matrix}\right ].
\]
When $q=1$ (the constant one polynomial),
the localizing matrix $L_{1}^{(k)}[y]$
reduces to a moment matrix, which we denote as
\[
M_k[y] \, := \, L_{1}^{(k)}[y].
\]
For instance, when $n=2$, $k=3$, the matrix $M_3[y]$ is
\[
M_3[y]=\left [\begin{matrix}
y_{00} & y_{10} & y_{01} & y_{20} & y_{11} & y_{02} & y_{30} & y_{21} & y_{12} & y_{03} \\
y_{10} & y_{20} & y_{11} & y_{30} & y_{21} & y_{12} & y_{40} & y_{31} & y_{22} & y_{13} \\
y_{01} & y_{11} & y_{02} & y_{21} & y_{12} & y_{03} & y_{31} & y_{22} & y_{13} & y_{04} \\
y_{20} & y_{30} & y_{21} & y_{40} & y_{31} & y_{22} & y_{50} & y_{41} & y_{32} & y_{23} \\
y_{11} & y_{21} & y_{12} & y_{31} & y_{22} & y_{13} & y_{41} & y_{32} & y_{23} & y_{14} \\
y_{02} & y_{12} & y_{03} & y_{22} & y_{13} & y_{04} & y_{32} & y_{23} & y_{14} & y_{05} \\
y_{30} & y_{40} & y_{31} & y_{50} & y_{41} & y_{32} & y_{60} & y_{51} & y_{42} & y_{33} \\
y_{21} & y_{31} & y_{22} & y_{41} & y_{32} & y_{23} & y_{51} & y_{42} & y_{33} & y_{24} \\
y_{12} & y_{22} & y_{13} & y_{32} & y_{23} & y_{14} & y_{42} & y_{33} & y_{24} & y_{15} \\
y_{30} & y_{13} & y_{04} & y_{23} & y_{14} & y_{05} & y_{33} & y_{24} & y_{15} & y_{06} \\
\end{matrix}\right ].
\]

In the following, we review semidefinite relaxations of
semi-algebraic sets. Consider the semi-algebraic set
\be \label{df:S}
S \, := \,
\{x\in {\mathbb{R}}^n: \,
g_1(x)\geq 0, \ldots, g_t(x) \geq 0\},
\ee
for polynomials $g_1, \ldots, g_t \in \re[x]$.
Denote the degrees
\be \label{def:d}
d_j := \lceil \deg(g_j)/2\rceil, \quad
d := \max_j d_j.
\ee
For all $k\ge d$ and for all $x\in S$, we have
\[
h_j(x) \big( [x]_{k-d_j } \big)
\big( [x]_{k-d_j } \big)^T  \succeq  0,
j = 1, \ldots, t.
\]
This implies that if $y = [u]_{2k}$ and $u \in S$, then
\[
L_{g_j}^{(k)}[y] \succeq 0,
j = 1, \ldots, t.
\]
Clearly, $[x]_{k}[x]_{k}^T\succeq 0$
for all $x \in \re^n$, so
\[
M_k[y] \succeq 0
\]
for all $y = [u]_{2k}$.
So, $S$ is always contained in the set
\be \label{SDr:Sk}
S_k := \left\{ x \in \re^n
\left| \baray{c}
\exists y \in \re^{\N_{2k}^n}, \, y_0 = 1, \, M_k[y] \succeq 0,  \\
x = (y_{e_1}, \ldots, y_{e_n} ),  \\
L_{g_j}^{(k)}[y] \succeq 0 \,( j=0, 1, \ldots, t) \\
\earay \right.
\right\},
\ee
for all $k \geq d$. Each $S_k$ is the projection of
a set in $\re^{\N_{2k}^n}$
that is defined by linear matrix inequalities.
It is a {\it semidefinite relaxation} of $S$,
because $S \subseteq S_k$ for all $k \geq d$.
It holds the nested containment relation
\be \label{nest:cont}
S \subseteq \cdots \subseteq S_{k+1} \subseteq S_k
\subseteq \cdots \subseteq S_d.
\ee

\section{A complete semidefinite algorithm}

We discuss how to detect copositivity of a given matrix or tensor.
Since matrix copositivity is a special case of tensor copositivity,
we only discuss the detection of copositive tensors.

For a symmetric tensor $\mA \in \mt{S}^m( \re^n )$,
let $\mA(x)$ be the homogeneous polynomial defined as in \reff{df:A(x)}.
Clearly, $\mA$ is copositive if and only if
$\mA(x) \geq 0$ for all $x$ belonging to the standard simplex
\[
\Dt \, = \, \{x\in \re^n:\, e^Tx=1,\, x\ge 0 \}.
\]
Consider the optimization problem
\be
\label{problem:simplex}
\left\{ \baray{rl}
v^* \,:= \, \min &\mA(x)  \\
\text{s.t.} &  e^Tx =1, \, (x_1, \ldots, x_n) \ge 0.
\earay \right.
\ee
Clearly, $\mA$ is copositive if and only if
the minimum value $v^* \geq 0$.
Therefore, testing the copositivity of $\mA$
is the same as determining the sign of $v^*$.
The \reff{problem:simplex} is a polynomial optimization problem.
A standard approach for solving it is to
apply classical Lasserre relaxations~\cite{Lasserre2001Global}.
Since the feasible set is compact and the archimedean condition holds,
its asymptotic convergence is always guaranteed.
However, there are still some issues in computation:
\bit

\item The convergence of classical Lasserre relaxations may be slow for some tensors.
Since the computational cost grows rapidly as
the relaxation order increases,
people often want faster convergence in practice.

\item For some tensors $\mA$, the sequence of classical Lasserre relaxations
might fail to have finite convergence. In other words,
it may require to solve infinitely many semidefinite relaxations,
to detect copositivity. This is not ideal in applications.

\item Certifying convergence of Lasserre's relaxations
is a critical issue in detecting copositivity.
The flat extension or truncation condition
is usually used for the certifying \cite{FlatTrun}.
However, it does not always hold, especially when
\reff{problem:simplex} has infinitely many minimizers.
For such cases, certifying convergence is mostly an open question.

\eit
In this section, we construct a new hierarchy of
semidefinite relaxations that can address the above issues.

Recently, as proposed in \cite{Nie2017Tight},
there exist tight relaxations for solving polynomial optimization,
whose constructions are based on optimality conditions and
Lagrange multiplier expressions.
Assume $u$ is an optimizer of \reff{problem:simplex}.
Then it satisfies the following optimality conditions
(the $\nabla$ denotes the gradient):
\be \label{cd:kkt}
\left\{ \baray{c}
\nabla \mA(u) \, = \,  \lambda_0 e + \sum_{i=1}^n \lambda_i e_i,  \\
\lmd_1 u_1 = \cdots = \lmd_nu_n = 0, \,\,
\lambda_1 \geq 0, \ldots, \lmd_n \geq 0,
\earay \right.
\ee
where $\lmd_0, \lmd_1, \ldots, \lmd_n$
are the Lagrange multipliers.
By a simple algebraic computation (also see \cite{Nie2017Tight}),
one can show that (note the identity $x^T \nabla \mA(x) = m \mA(x)$)
\be
\left\{ \baray{rcl}
\lmd_0 &=& u^T \nabla \mA(u) = m \mA(u), \\
\lmd_i &=& \frac{\partial \mA(u)}{\partial x_i}- m\mA(u)
\,\,\, (i=1,2,\ldots,n).
\earay \right.
\ee
Because of the above expressions,
we define new polynomials:
\be
\left\{ \baray{rcl}
p_0 & :=&   m \mA(x), \\
p_i &:=& \frac{\partial \mA(x)}{\partial x_i}- m\mA(x)
\,\,\, (i=1,2,\ldots,n).
\earay \right.
\ee
Since every optimizer $u$ must satisfy \reff{cd:kkt}
and the norm $\|u\| \leq 1$,
the optimization problem \reff{problem:simplex} is equivalent to
\be
\label{problem:polynomial multipler}
\left\{ \baray{rl}
\min &  \mA(x)  \\
\text{s.t} &  e^Tx-1=p_1(x) x_1=\cdots = p_n(x)x_n = 0, \\
        &  1-\|x\|^2 \geq 0,\,  x_i \geq 0,\,   p_i(x)\ge 0 \, (i=1,\ldots,n).
\earay \right.
\ee
Then we apply Lasserre's relaxations to solve
\reff{problem:polynomial multipler}.
For the orders $k=1,2,\ldots$,
solve the semidefinite relaxation problem
\be
\label{problem:Lasserre's relaxation}
\left\{ \baray{rl}
v_k \, := \, \min  &   \langle \mA(x),y \rangle  \\
 \text{s.t} &   L^{(k)}_{e^Tx-1}[y]=0, \, L^{(k)}_{x_i p_i}[y]=0 \, (i=1,\ldots,n),  \\
        & L^{(k)}_{1-\|x\|^2}[y] \succeq 0, \, L^{(k)}_{x_i}[y] \succeq 0, \,
          L^{(k)}_{p_i}[y] \succeq 0 \, (i=1,\ldots,n),  \\
        & y_0 =1, M_k[y] \succeq 0, \, y\in \re^{\N_{2k}^n}.
\earay \right.
\ee
Note that $v^*$ is also the optimal value of
\reff{problem:polynomial multipler}. From the nested
relation \reff{nest:cont}, the feasible set of \reff{problem:polynomial multipler}
is contained in the projection of that of \reff{problem:Lasserre's relaxation},
so the optimal value $v_k$ of
\reff{problem:Lasserre's relaxation} satisfies
\[
v_1 \leq v_2 \leq \cdots \leq v^*.
\]
Clearly, if $v_k \geq 0$ for some $k$,
then $\mA$ is copositive.
Combining the above, we can get the following algorithm.

\begin{alg}
\label{alg:cop}
For a given tensor $\mA \in \mt{S}^m( \re^n )$,
let $m_0 := \lceil m/2 \rceil$.
We test its copositivity by doing the following:
\bit

\item [Step~0:] Generate a generic vector $\xi \in \re^{\N_{m}^n}$.
Let $k:= m_0$.

\item [Step~1:] Solve the semidefinite relaxation~\reff{problem:Lasserre's relaxation}.
If its optimal value $v_k \geq 0$,
then $\mA$ is copositive and stop.
If $v_k < 0$, go to Step~2.

\item [Step~2:] Solve the following semidefinite program
\be \label{las:v-Ax>=0}
\left\{ \baray{rl}
 \min  &   \langle \xi^T[x]_m,y \rangle  \\
 \text{s.t} &   L^{(k)}_{e^Tx-1}[y]=0, \,  L^{(k)}_{x_i}[y] \succeq 0, \, (i\in [n]),  \\
        & L^{(k)}_{1-\|x\|^2}[y] \succeq 0, \,
          L^{(k)}_{v_k - \mA(x)}[y] \succeq 0,      \\
        & y_0 =1, M_k[y] \succeq 0, \, y\in \re^{\N_{2k}^n},
\earay \right.
\ee
for an optimizer $\hat{y}$ if it is feasible.
If it is infeasible, let $k:=k+1$ and go to Step~1.

\item [Step~3:]  Let $u = \big( (\hat{y})_{e_1}, \ldots, (\hat{y})_{e_n} \big)$.
If $\mA(u) < 0$, then $\mA$ is not copositive and stop;
otherwise, let $k:=k+1$ and go to Step~1.

\eit

\end{alg}

In Step~2, the copositivity of $\mA$ is justified
by the relationship $v^* \geq v_k$, for all $k\geq m_0$.
In Step~3, the point $u$ must satisfy the constraint in \reff{problem:simplex}.
This is because of the constraints
$L^{(k)}_{e^Tx-1}[y]=0$
and $L^{(k)}_{x_i}[y] \succeq 0$.
In the following, we show that Algorithm~\ref{alg:cop}
must terminate within finitely many iterations,
for all tensors $\mA$. In other words,
the copositivity of every $\mA$ can be detected correctly
by solving finitely many semidefinite relaxations.
These properties are summarized as follows.

\begin{theorem}
\label{thm:alg:cvg}
For all symmetric tensors $\mA \in \mt{S}^m(\re^n)$,
Algorithm~\ref{alg:cop} has the following properties:
\bit

\item [(i)] For all $k \geq m_0$,
the semidefinite relaxation \reff{problem:Lasserre's relaxation}
is feasible and achieve its optimal value $v_k$;
moreover, $v_k = v^*$ for all $k$ sufficiently large.

\item [(ii)] For all $k \geq m_0$,
the semidefinite program \reff{las:v-Ax>=0}
has an optimizer if it is feasible.

\item [(iii)] If $\mA$ is copositive, then Algorithm~\ref{alg:cop} must stop
with $v_k \geq 0$, when $k$ is sufficiently large.

\item [(iv)] If $\mA$ is not copositive,
then, for almost all $\xi \in \re^{ \N_m^n }$,
Algorithm~\ref{alg:cop} must stop
with $f(u) < 0$ and $u \in \Dt$,
when $k$ is sufficiently large.

\eit

\end{theorem}
\begin{proof}
(i) The feasible set of \reff{problem:simplex}
is compact, so it must have a minimizer, say, $u^*$.
Then, $u^*$ satisfies \reff{cd:kkt},
and hence $u^*$ is a feasible point for \reff{problem:polynomial multipler}.
So, the feasible set of \reff{problem:polynomial multipler} is nonempty.
This implies that the semidefinite relaxation~\reff{problem:Lasserre's relaxation}
is always feasible. By the constraint $L^{(k)}_{1-\|x\|^2}[y] \succeq 0$,
we can show that the feasible set of \reff{problem:Lasserre's relaxation}
is compact, as follows. First, we can see that
\[
1 = y_0 \geq y_{2e_1} + \cdots + y_{2e_n}.
\]
So, $0 \leq y_{2e_i} \leq 1$ since each $y_{2e_i} \geq 0$
(because $M_k[y]\succeq 0$).
Second, for all $0< | \af | \leq k-1$,
the $(\af,\af)$th diagonal entry of $L^{(k)}_{1-\|x\|^2}[y]$
is nonnegative, so
\be  \label{1-x2:y:>=0}
y_{2\af} \geq y_{2\af+2e_1} + \cdots + y_{2\af+2e_n}.
\ee
By choosing $\af = e_1,\ldots, e_n$, the same argument
can show that $0 \leq y_{2\bt} \leq 1$ for all $|\bt| \leq 2$.
By repeatedly applying \reff{1-x2:y:>=0},
one can further get that $0 \leq y_{2\bt} \leq 1$ for all $|\bt| \leq k$.
Third, note that the diagonal entries
of $M_k[y]$ are precisely $y_{2\bt}$ with $|\bt| \leq k$.
Since $M_k[y] \succeq 0$,
all the entries of $M_k[y]$ must be between $-1$ and $1$.
This means that $y$ is bounded, hence the feasible set of
\reff{problem:Lasserre's relaxation} is compact.
Therefore, \reff{problem:Lasserre's relaxation}
must achieve its optimal value $v_k$.
To prove $v_l = v^*$ for all $k$ sufficiently large, note that
\reff{problem:polynomial multipler} is the same as the optimization
\be
\label{pop:noball}
\left\{ \baray{rl}
\min &  \mA(x)  \\
\text{s.t} &  e^Tx-1=p_1(x) x_1=\cdots = p_n(x)x_n = 0, \\
        &    x_i \geq 0,\,   p_i(x)\ge 0, \, i =1,\ldots, n.
\earay \right.
\ee
Its corresponding Lasserre's relaxations are
\be
\label{las:noball}
\left\{ \baray{rl}
v_k^{\prime} \,:=\, \min  &   \langle \mA(x),y \rangle  \\
 \text{s.t} &   L^{(k)}_{e^Tx-1}[y]=0, \, L^{(k)}_{x_i p_i}[y]=0 \, (1 \le i \le n),  \\
        &  L^{(k)}_{x_i}[y] \succeq 0, \,
          L^{(k)}_{p_i}[y] \succeq 0 \, (1 \le i \le n),  \\
        & y_0 =1, M_k[y] \succeq 0, \, y\in \re^{\N_{2k}^n},
\earay \right.
\ee
for the orders $k=1,2,\ldots$.
The optimal value of \reff{pop:noball} is also $v^*$.
The feasible set of \reff{problem:Lasserre's relaxation}
is contained in that of \reff{las:noball}, so
\be \label{vkp<=vk}
v_k^{\prime}  \leq v_k \leq v^*, \quad k= m_0, m_0+1, \ldots.
\ee
Next, we show that the set of polynomials
\[
F \, := \, \left\{
(1-e^Tx)\phi + \sum_{j=1}^n x_j \big( \sum_\ell s_{j,\ell}^2 \big):\,
\phi \in \re[x], \, s_{j,\ell} \in \re[x]
\right\}
\]
is archimedean, i.e., there exists $f \in F$ such that
the inequality $f(x) \geq 0$ defines a compact set in $\re^n$.
This is true for $f=1-\|x\|^2$, because of the identity
\be
 1-\|x\|^2 =(1-e^Tx)(1+\|x\|^2)+\sum_{i=1}^n x_i(1-x_i)^2+\sum_{i\neq j}x_i^2x_j.
\ee
%
%
By Theorem 3.3 of \cite{Nie2017Tight}, we know that
$v_k^{\prime} = v^*$ when $k$ is sufficiently large.
Hence, the relation \reff{vkp<=vk} implies that
$v_k = v^*$ for all $k$ sufficiently large.

(ii) The semidefinite program \reff{las:v-Ax>=0}
also has the constraint $L^{(k)}_{1-\|x\|^2}[y] \succeq 0$.
By the same argument as in (i),
we know that the feasible set of
\reff{las:v-Ax>=0} is compact.
So, it must have an optimizer if it is feasible.

(iii) Clearly, $\mA$ is copositive if and only if
$v^* \geq 0$. By the item (i), $v_k = v^*$ for all $k$ big enough.
Therefore, if $\mA$ is copositive,
we must have $v_k \geq 0$ for all $k$ large enough.

(iv) If $\mA$ is not copositive, then $v^* < 0$.
By the item (i), there exists $k_1 \in \N$ such that
$v_k = v^*$ for all $k \geq k_1$.
Hence, for all $k \geq k_1$, \reff{las:v-Ax>=0}
is the same as
\be \label{las:v-Ax>=0:const}
\left\{ \baray{rl}
 \min  &   \langle \xi^T[x]_m,y \rangle  \\
 \text{s.t} &   L^{(k)}_{e^Tx-1}[y]=0, \,  L^{(k)}_{x_i}[y] \succeq 0, \, (i\in [n]),  \\
        & L^{(k)}_{1-\|x\|^2}[y] \succeq 0, \,
          L^{(k)}_{v^* - \mA(x)}[y] \succeq 0,      \\
        & (y)_0 =1, M_k[y] \succeq 0, \, y\in \re^{\N_{2k}^n}.
\earay \right.
\ee
It is the $k$th Lasserr's relaxation for the polynomial optimization
\be \label{pop:xi}
\left\{ \baray{rl}
 \min  &   \xi^T[x]_m  \\
 \text{s.t} & e^Tx-1 = 0,\, x \geq 0,
v^* - \mA(x) \geq 0.
\earay \right.
\ee
The feasible set of \reff{pop:xi}
is clearly compact. When $\xi$ is generically chosen in
$\re^{ \N_m^n }$, \reff{pop:xi} has a unique optimizer, say, $u^*$.
Hence, for almost all $\xi \in \re^{ \N_m^n }$,
$u^*$ is the unique optimizer.
For notation convenience, denote by
$\hat{y}^k$ the optimizer of \reff{las:v-Ax>=0}
with the relaxation order $k$.
Let $u^k = \big( (\hat{y}^k)_{e_1}, \ldots, (\hat{y}^k)_{e_n} \big)$.
By Corollary~3.5 of \cite{Swg05} or Theorem~3.3 of \cite{FlatTrun},
the sequence $\{u^k\}_{k=m_0}^{\infty}$
must converge to $u^*$, the unique optimizer of \reff{pop:xi}.
Since $\mA(u^*) \leq v^* < 0$, we must have
$\mA(u^k) < 0$ when $k$ is sufficiently large.
Moreover, the constraints $L^{(k)}_{x_i}[y] \succeq 0$
imply that $u^k \geq 0$, and
$L^{(k)}_{e^Tx-1}[y]=0$ implies that $e^Tu^k=1$.
Therefore, $u^k \in \Dt$.
\end{proof}

\begin{remark} \label{rmk:vk>=0}
In Step~1 of Algorithm~\ref{alg:cop}, we need to
test whether or not $v_k \geq 0$.
When the absolute value of $v_k$ is big,
this testing is very easy. However, if its absolute value
is very small, then testing the sign might be difficult.
Note that the semidefinite relaxation~\reff{problem:Lasserre's relaxation}
is solved numerically, i.e., $v_k$ is accurate up to a tiny round-off error.
This difficulty is not because of theoretical properties of Algorithm~\ref{alg:cop},
but due to round-off errors, which occur in all numerical methods.
In practice, if $v_k$ is positive or close to zero
(say, $v_k > - 10^{-6}$), then it is reasonably well to claim
that $\mA$ is copositive.
\end{remark}

\section{Numerical Experiments}

This section presents numerical experiments of
applying Algorithm~\ref{alg:cop} to detect matrix and tensor copositivity.
The computation is implemented in MATLAB R2016b,
on a Lenovo Laptop with CPU@2.90GHz and RAM 16.0G.
Algorithm~\ref{alg:cop} can be implemented by using the software
{\tt Gloptipoly~3} \cite{Gloptipoly},
which calls the semidefinite program solver {\tt SeDuMi} \cite{sedumi}.
For convenience of presentation, we only display $4$ decimal digits.
Consumed time in computation is reported in seconds (s).
Recall that $v_k$ is the optimal value of \reff{problem:Lasserre's relaxation}.
We refer to Remark~\ref{rmk:vk>=0}
for how to test the sign condition $v_k \geq 0$.

First, we see some copositive matrices that is not a sum of
psd and nonnegative matrices.

\begin{example}
\label{Example:Horn}
Consider the Horn matrix \cite{HalNew63}
\be \label{mat:horn}
    H := \begin{bmatrix}
    \begin{array}{rrrrr}
    1 & -1 & 1 & 1 & -1 \\
    -1 & 1 & -1 & 1 & 1 \\
    1 & -1 & 1 & -1 & 1 \\
    1 & 1 & -1 & 1 & -1 \\
    -1 & 1 & 1 & -1 & 1
    \end{array}
\end{bmatrix}.
\ee
It is known that $H$ is copsoitive
but is not a sum of psd and nonnegative matrices.
We apply Algorithm~\ref{alg:cop} to test its copositivity.
The lower bounds $v_k$ are shown in Table~\ref{cptab:Horn}.
\bcen
\begin{table}[htb]
\caption{
Comp. results for the matrix in Example~\ref{Example:Horn}
}
\begin{tabular}{|c|c|c|c|}  \hline
 $k$        &     1       &     2       &         3       \\  \hline
 $v_k$    &  $-0.7889$  &  $-0.0472$  &   $-7.0 \times 10^{-8}$  \\  \hline
 time(s)          &   $0.59$    &   $0.35$    &   $1.68$     \\ \hline
\end{tabular}
\label{cptab:Horn}
\end{table}
\ecen
For $k=3$, we can conclude the copositivity of $H$,
up to a tiny round-off error.
\end{example}

\begin{example}
\label{mat:Hof:Per}
Consider the Hoffman-Pereira matrix
\be
P \, := \,
\left[\begin{array}{rrrrrrr}
     1 & -1 & 1 & 0 & 0 & 1 & -1 \\
     -1 & 1 & -1 & 1 & 0 & 0 & 1 \\
     1 & -1 & 1 & -1 & 1 & 0 & 0 \\
     0 & 1 & -1 & 1 & -1 & 1 & 0 \\
     0 & 0 & 1 & -1 & 1 & -1 & 1 \\
     1 & 0 & 0 & 1 & -1 & 1 & -1 \\
     -1 & 1 & 0 & 0 & 1 & -1 & 1 \\
\end{array} \right].
\ee
Like the Horn matrix, $P$ is also copositive but not a sum of
psd and nonnegative matrices \cite{hoffman1973copositive}.
The computational results by Algorithm~\ref{alg:cop}
are shown in Table~\ref{cptab:HofPer}.
\bcen
\begin{table}[htb]
\caption{
Comp. results for the matrix in Example~\ref{mat:Hof:Per}
}
\begin{tabular}{|c|c|c|c|}  \hline
 $k$      &     1       &     2       &         3       \\  \hline
 $v_k$    &   $-0.4503$   &   $-0.0250$  &   $-2.2 \times 10^{-7}$ \\  \hline
 time(s)     &    $0.58$   &   $0.60$   &  $24.85$  \\ \hline
\end{tabular}
\label{cptab:HofPer}
\end{table}
\ecen
The copositivity of $P$ is detected when $k=3$,
up to a tiny round-off error.
\end{example}

\begin{example}
\label{mat:Hild}
Consider the matrix $B$ that is given as
\[
\left[\begin{array}{ccccc}
 1   & -\cos \psi_4 & \cos(\psi_4+\psi_5) & \cos(\psi_2+\psi_3) & -\cos \psi_3\\
-\cos \psi_4 &  1 &  -\cos \psi_5 & \cos(\psi_1+\psi_5) & \cos(\psi_3+\psi_4)\\
\cos(\psi_4+\psi_5) & -\cos \psi_5 & 1 & -\cos \psi_1&  \cos(\psi_1+\psi_2)\\
\cos(\psi_2+\psi_3) & \cos(\psi_1+\psi_5) & -\cos \psi_1& 1 & -\cos \psi_2 \\
 -\cos \psi_3 & \cos(\psi_3+\psi_4)&\cos(\psi_1+\psi_2)& -\cos \psi_2 &1
\end{array} \right],
\]
where each $\psi_i \ge 0$ and $\sum_{i=1}^5 \psi_i <\pi$.
The matrix $B$ gives an extreme ray of the cone $\mc{COP}_n$
but is not a sum of psd and nonnegative matrices
\cite{hildebrand2012extreme}. For convenience, we test its copositivity for the case
\[
\psi_1 =  \psi_2 =  \psi_3 =  \psi_4 =  \psi_5 =  \pi/6.
\]
By Algorithm~\ref{alg:cop}, we get
the computational results in Table~\ref{cptab:Hild}.
\bcen
\begin{table}[htb]
\caption{
Comp. results for the matrix in Example~\ref{mat:Hild}
}
\begin{tabular}{|c|c|c|c|}  \hline
 $k$      &     1       &     2      &         3       \\  \hline
 $v_k$    &  $-0.2218$ & $-0.0153$   &   $-1.2 \times 10^{-8}$  \\  \hline
 time(s)  &  $0.61$  &  $0.32$   &   $1.11$   \\  \hline
\end{tabular}
\label{cptab:Hild}
\end{table}
\ecen
The copositivity is detected for $k=3$,
up to a tiny round-off error.
\end{example}

Next, we see a matrix that is not copositive.

\begin{example}
Consider the following matrix
\be \label{mat:horn99}
   A := \begin{bmatrix}
    \begin{array}{rrrrr}
    1 & -1 & 1 &  1 & -1 \\
    -1 & 1 & -1 & 1 & 1 \\
    1 & -1 & 1 & -1 & 1 \\
    1 & 1 & -1 & 1 & -1 \\
    -1 &  1 & 1 & -1 & 0.99
    \end{array}
\end{bmatrix}.
\ee
It is obtained from the Horn matrix, by subtracting
$0.01$ from the $(5,5)$-entry.
We apply Algorithm~\ref{alg:cop} to test copositivity.
It terminates at $k=3$, with the point
\[
u \,=\, ( 0.4474,\,   0.0000,\,    0.0000,\,    0.0513,\,    0.5012)
\]
that refutes the copositivity,
because $u^TAu \approx  -0.0025 < 0$.
\end{example}

\begin{example}
\label{Mot:Rob:Choi}
We consider three tensors $\mA \in \mt{S}^3(\re^3)$
whose polynomials $\mA(x)$ are respectively given as
\be
\left\{ \baray{rl}
 \text{Motzkin: } & \mA(x) := x_1^2x_2+x_1x_2^2+x_3^3-3x_1x_2x_3, \\
 \text{Robinson: }& \mA(x) := x_1^3+x_2^3+x_3^3-x_1^2x_2-x_1x_2^2-x_1^2x_3 \\
                  & \qquad \qquad -x_1x_3^2-x_2^2x_3-x_2x_3^2+3x_1x_2x_3, \\
\text{Choi-Lam: } & \mA(x) := x_1^2x_2+x_2^2x_3+x_3^2x_1-3x_1x_2x_3.
\earay \right.
\ee
When each $x_i$ is replaced by $x_i^2$, the polynomials $\mA(x)$
are respectively the Motzkin, Robinson and Choi-Lam polynomials
(they are all nonnegative but not sum of squares \cite{Rez00}).
Hence, these tensors are all copositive.
We detect their copositivity by Algorithm~\ref{alg:cop}.

The computational results are shown in Table~\ref{cptab:Choi-Lam}.
For all these tensors, the copositivity is confirmed for $k=3$,
up to tiny round-off errors.
\bcen
\begin{table}[htb]
\caption{
Comp. results for the tensors in Example~\ref{Mot:Rob:Choi}
}
\begin{tabular}{|c|c|c|c|c|c|c|}  \hline
     $\mA(x)$   &   \multicolumn{2}{c|}{Motzkin}     &
 \multicolumn{2}{c|}{Robinson}  &  \multicolumn{2}{c|}{Choi-Lam}   \\  \hline
 $k$   & $v_k$ & time(s) & $v_k$ & time(s) & $v_k$ & time(s) \\ \hline
  $2$  &  $-0.0045$  & 0.78   & $-0.0208$ & 0.76 & $-0.0129$ & 0.77  \\ \hline
  $3$  &  $-4.3 \times 10^{-8}$  & 0.45
       &  $-4.9 \times 10^{-8}$  & 0.23
       &  $-2.1 \times 10^{-8}$  & 0.37  \\ \hline
\end{tabular}
\label{cptab:Choi-Lam}
\end{table}
\ecen

\end{example}

\begin{example}
\label{Example:x4-16x2}
Consider the tensor $\mA \in \mt{S}^4(\re^4)$ that is given as
\[
\mA_{1     1     1     1} =     1,\,
\mA_{1     1     1     2} =     1,\,
\mA_{1     1     1     3} =     1,\,
\mA_{1     1     1     4} =     1,\,
\mA_{1     1     2     2} =     1,\,
\mA_{1     1     2     3} =     1,\,
\mA_{1     1     2     4} =     1,\,
\]
\[
\mA_{1     1     3     3} =     1,\,
\mA_{1     1     3     4} =     1,\,
\mA_{1     1     4     4} =     1,\,
\mA_{1     2     2     2} =     1,\,
\mA_{1     2     2     3} =    -3,\,
\mA_{1     2     2     4} =     1,\,
\mA_{1     2     3     3} =     1,\,
\]
\[
\mA_{1     2     3     4} =    -3,\,
\mA_{1     2     4     4} =     1,\,
\mA_{1     3     3     3} =     1,\,
\mA_{1     3     3     4} =     1,\,
\mA_{1     3     4     4} =     1,\,
\mA_{1     4     4     4} =     1,\,
\mA_{2     2     2     2} =     1,\,
\]
\[
\mA_{2     2     2     3} =     1,\,
\mA_{2     2     2     4} =     1,\,
\mA_{2     2     3     3} =     1,\,
\mA_{2     2     3     4} =     1,\,
\mA_{2     2     4     4} =     1,\,
\mA_{2     3     3     3} =     1,\,
\mA_{2     3     3     4} =    -3,\,
\]
\[
\mA_{2     3     4     4} =     1,\,
\mA_{2     4     4     4} =     1,\,
\mA_{3     3     3     3} =     1,\,
\mA_{3     3     3     4} =     1,\,
\mA_{3     3     4     4} =     1,\,
\mA_{3     4     4     4} =     1,\,
\mA_{4     4     4     4} =     1.
\]
One can verify that
\[
\mA(x) \,=\, (x_1+x_2+x_3+x_4)^4-16(x_1x_2+x_2x_3+x_3x_4)^2.
\]
This tensor is copositive, because $\mA(x)$ has the factorization
\[
\big( (x_1-x_2+x_3-x_4)^2+4x_1x_4 \big)
\cdot \big( (x_1+x_2+x_3+x_4)^2+4(x_1x_2+x_2x_3+x_3x_4)  \big).
\]
For $k=2$, we get $v_2 \approx -0.3862$;
it took about $0.8$ second.
For $k=3$, we get $v_3 \approx -1.4 \times 10^{-7}$.
It took about $0.6$ second.
The copositivity is detected at $k=3$, up to a round-off error.
\end{example}

\begin{example}
\label{Example:Clique}
Copositive matrices have important applications in graph theory.
For a graph $G$, the maximum size of its complete subgraph
is called the clique number of $G$, denotes $\gamma(G)$.
Let $A$ be the adjacency matrix of $G$
and $E$ be the matrix of all ones.
It can be shown that \cite{DeKlerk2002Approximation}
\be
\gamma(G)=\min \{
\lambda: \, \lambda(E-A)-E \,\text{ is copositive}
\}.
\ee
Therefore, we can determine the clique number of a graph
by checking copositivity. For instance,
consider the graph $G$ with the following adjacency matrix
\be
A =\left(
     \begin{array}{cccccccc}
        0 & 1 & 0 & 1 & 1 & 0 & 0 & 1 \\
        1 & 0 & 0 & 1 & 0 & 1 & 1 & 1 \\
        0 & 0 & 0 & 0 & 0 & 0 & 0 & 0 \\
        1 & 1 & 0 & 0 & 1 & 0 & 1 & 0 \\
        1 & 0 & 0 & 1 & 0 & 1 & 1 & 1 \\
        0 & 1 & 0 & 0 & 1 & 0 & 0 & 1 \\
        0 & 1 & 0 & 1 & 1 & 0 & 0 & 1 \\
        1 & 1 & 0 & 0 & 1 & 1 & 1 & 0 \\
     \end{array}
 \right).
\ee
One can check that its clique number $\gamma(G)=3$,
so the matrix $B:=3(E-A)-E$ is copositive.
We apply Algorithm~\ref{alg:cop} to detect the copositivity.
For $k=1$,  $v_1 \approx -1.7039$; it took about $0.6$ second.
For $k=2$, $v_2 \approx  -1.6 \times 10^{-7}$; it took about $1.6$ seconds.
The copositivity is detected when $k=2$.
\end{example}

\begin{example}
\label{Example:HyperG}
Consider the tensor $\mA \in \mt{S}^3(\re^5)$ given such that
\begin{eqnarray*}
        \mathcal{A}(:,:,1)=
        \left(
     \begin{array}{ccccc}
        1 & 1 & 0 & 1 & 1 \\
        1 & 1 & 0 & 0 & 1 \\
        0 & 0 & 0 & 0 & 0 \\
        1 & 0 & 0 & 0 & 0 \\
        1 & 1 & 0 & 0 & 1 \\
     \end{array}
 \right), &
 \mathcal{A}(:,:,2)=
 \left(
     \begin{array}{ccccc}
        1 & 1 & 0 & 0 & 1 \\
        1 & 0 & 0 & 0 & 1 \\
        0 & 0 & 1 & 0 & 0 \\
        0 & 0 & 0 & 0 & 1 \\
        1 & 1 & 0 & 1 & 0 \\
     \end{array}
 \right), \\
 A(:,:,3)=
 \left(
     \begin{array}{ccccc}
        0 & 0 & 0 & 0 & 0 \\
        0 & 0 & 1 & 0 & 0 \\
        0 & 1 & 0 & 0 & 0 \\
        0 & 0 & 0 & 1 & 1 \\
        0 & 0 & 0 & 1 & 0 \\
     \end{array}
 \right), &
 A(:,:,4)=\left(
     \begin{array}{ccccc}
        1 & 0 & 0 & 0 & 0 \\
        0 & 0 & 0 & 0 & 1 \\
        0 & 0 & 0 & 1 & 1 \\
        0 & 0 & 1 & 0 & 0 \\
        0 & 1 & 1 & 0 & 0 \\
     \end{array}
 \right), \\
 A(:,:,5)=\left(
     \begin{array}{ccccc}
        1 & 0 & 0 & 0 & 0 \\
        0 & 0 & 0 & 0 & 1 \\
        0 & 0 & 0 & 1 & 1 \\
        0 & 0 & 1 & 0 & 0 \\
        0 & 1 & 1 & 0 & 0 \\
     \end{array}
 \right).
\end{eqnarray*}
The tensor $\mA$ is clearly copositive, since its entries are nonnegative.
We consider the new tensor of the form
\[
\mathcal{H}=\rho (\mathcal{I}+\mathcal{A})-\mathcal{E},
\]
with $\rho$ a parameter. Here, $\mc{E}$ is the tensor of all ones,
and $\mc{I}$ is the tensor such that
$\mc{I}_{ijk}=1$ if $i=j=k$ and $\mc{I}_{ijk}=0$ otherwise.
The copositivity of tensors like
$\mc{H}$ is important in estimating coclique numbers of hypergraphs
\cite{Chen2016Copositive}.
For a range of values of the parameter $\rho$,
the computational results are shown in Table~\ref{cptab:hypG}.
\bcen
\begin{table}[htb]
\caption{
Comp. results for the matrix in Example~\ref{Example:HyperG}
}
\begin{tabular}{|c|r|c|c|}  \hline
$\rho$  &  $v_2$ \qquad \qquad    &  time(s)  & copositivity   \\  \hline
4.400   &  $1.1 \times 10^{-2}$   &   0.15   &  yes    \\  \hline
4.353   &  $3.2 \times 10^{-4}$   &   0.15   &  yes  \\   \hline
4.352   &  $9.8 \times 10^{-5}$   &   0.18   &  yes  \\   \hline
4.351   &  $-1.3 \times 10^{-4}$  &   0.16   &  no   \\   \hline
4.350   &  $-3.6 \times 10^{-4}$  &   0.16   &  no   \\   \hline
4.300   &  $-1.1 \times 10^{-2}$  &   0.15   &  no   \\   \hline
\end{tabular}
\label{cptab:hypG}
\end{table}
\ecen
In all the computation, the order $k=2$ is enough for detecting copositivity.
Because of the monotonicity of $\mc{H}$ in $\rho$,
we can also conclude that $\mc{H}$ is copositive for $\rho \geq 4.4$
and not copositive for $\rho \leq 4.3$.
\end{example}

\section{Conclusions and discussions}

This paper gives a complete semidefinite algorithm for
detecting copositive of matrices and tensors.
If the matrix or tensor $\mA$ is copositive,
we can get a certificate for that, i.e.,
a nonnegative lower bound for the optimal value
$v^*$ of \reff{problem:simplex}.
If it is not copositive, we can get a point that refutes the copositivity,
i.e., a point $u \in \Dt$ such that $\mA(u) <0$. For all $\mA$,
the copositivity can be detected by solving a finite number of
semidefinite relaxations. This property is shown in Theorem~\ref{thm:alg:cvg}.

\subsection*{Completely positive tensors and matrices}

A symmetric tensor $\mA \in \mt{S}^m(\re^n)$ is {\it completely positive}
if there exists vectors $u_1, \ldots, u_r \in \re_+^n$ such that
\[
\mA \, = \, (u_1)^{\otimes m} + \cdots + (u_r)^{\otimes m}.
\]
When the order $m=2$, this gives the definition of completely positive matrices.
We denote by $\mc{CP}_n$ the cone of $n$-by-$n$ completely positive matrices,
and $\mc{CP}_{m,n}$ the cone of completely positive tensors in $\mt{S}^m(\re^n)$.
The cone $\mc{CP}_n$ is dual to $\mc{COP}_n$
and $\mc{CP}_{m,n}$ is dual to $\mc{COP}_{m,n}$,
under the standard Frobenius norm product.
The problem of checking whether or not a matrix is completely positive
is NP-hard \cite{DicGij14}. This is also true for detecting completely positive tensors.
We refer to \cite{Bom00,DicDur12,ATKMP,SpoDur14}
for recent work on completely positive matrices,
and refer to \cite{FanZhou17,LuoQi16,QXX14}
for completely positive tensors.
A survey can be found in \cite{BerSha03,BerDurSha15}.

\subsection*{Copositive programming}

A basic linear copositive optimization problem is
\be \label{linop:cop}
\left\{\baray{rl}
\min & \mbox{trace}(CX) \\
s.t. & \mbox{trace}(A_iX) = b_i \, (i=1,\ldots,\ell), \\
 &  X \in \mc{COP}_n,
\earay \right.
\ee
where $A_1, \ldots, A_\ell, C$ are given $n$-by-$n$ symmetric matrices
and $b_1,\ldots, b_\ell$ are given reals.
The dual optimization problem of \reff{linop:cop} is
\be \label{dualop:cp}
\left\{\baray{rl}
\max &  \sum_{i=1}^\ell b_i y_i \\
s.t. & C - \sum_{i=1}^\ell y_i A_i \in \mc{CP}_n.
\earay \right.
\ee
When there is no objective in \reff{linop:cop} (resp., \reff{dualop:cp}),
it is reduced to a feasibility problem about copositive (resp., completely positive)
under linear constraints.
We refer to \cite{Bomz12,Dur10} for surveys in the area.
The linear conic optimization about the copositive
(or completely positive) cone is a special case of
linear optimization with moment (or nonnegative polynomial) conic constraints,
which was discussed in \cite{linopt}.

\subsection*{Comparison with classical Lasserre relaxations}

A symmetric tensor (or matrix) $\mA$ is copositive if and only if
the optimal value $v^*$ of \reff{problem:simplex} is greater than
or equal to zero. Note that \reff{problem:simplex} is
a polynomial optimization problem.
The classical Lasserre's hierarchy of semidefinite relaxations \cite{Lasserre2001Global}
can be applied to solve it. For a relaxation order $k$,
the Lasserre relaxation for solving \reff{problem:simplex} is
\be
\label{stdLas:kth}
\left\{ \baray{rl}
 \min  &   \langle \mA(x),y \rangle  \\
 \text{s.t} &   L^{(k)}_{e^Tx-1}[y]=0, \, L^{(k)}_{1-x'x}[y] \succeq 0 \,   \\
        &  L^{(k)}_{x_i}[y] \succeq 0 \, (1 \leq i \leq n),  \\
        & y_0 =1, M_k[y] \succeq 0, \, y\in \re^{\N_{2k}^n}.
\earay \right.
\ee
Let $\nu_k$ be the optimal value of \reff{stdLas:kth}.
Since the feasible set is compact and the archimedean condition holds,
one can also show that $\nu_k \to v^*$ as $k \to \infty$. However,
compared with the relaxation~\reff{problem:Lasserre's relaxation}
used by Algorithm~\ref{alg:cop}, \reff{stdLas:kth} is weaker.
This is because the feasible set of
\reff{problem:Lasserre's relaxation} is contained in that of \reff{stdLas:kth}.
So, $\nu_k \leq v_k \leq v^*$ for all relaxation orders $k$.
Here, we give an example of comparing the lower bounds $\nu_k, v_k$.
Consider the tensor in Example~\ref{Example:x4-16x2}.
The computational results are compared in Table~\ref{compare:x4-16x2}.
The optimal value $v^*=0$.
\begin{table}[htb]
    \centering
\caption{
A comparision of relaxations
\reff{problem:Lasserre's relaxation} and \reff{stdLas:kth}
for the tensor in Example~\ref{Example:x4-16x2}.
}
\begin{tabular}{|c|c|c|c|c|} \hline
\multirow{2}{*}{$k$} & \multicolumn{2}{|c|}{ relaxation~\reff{problem:Lasserre's relaxation} }
    & \multicolumn{2}{|c|}{ relaxation~\reff{stdLas:kth} } \\
\cline{2-5}  & time & $v_k$  & time & $\nu_k$ \\  \hline
2  & 0.8306  & $-0.3862$  &  0.2274   &  $-0.3862$  \\  \hline
3  & 0.5528  &  $-1.4 \times 10^{-7}$   &  0.4428  &  $-0.0010$ \\  \hline
4  & 1.5565  &  $-3.0 \times 10^{-7}$   &  1.8834  &  $-0.0002$  \\ \hline
5  & 8.0350  &  $-3.7 \times 10^{-7}$   &  10.8022  &  $-0.0001$  \\ \hline
\end{tabular}
\label{compare:x4-16x2}
\end{table}
For $k=2$, $v_k =\nu_k$; but for $k=3,4,5$, $v_k \gg \nu_k$.
Indeed, Algorithm~\ref{alg:cop} terminates
at $k=3$, and the coposivity is detected.
As a comparison, the convergence of $\nu_k$ to $v^*$ is slower.

\subsection*{Comparison with other algorithms based on simplicial partition}
There also exists algorithms for detecting copositivity
based on simplitical partition, such as the work
\cite{Bundfuss2008Algorithmic,CHQ17,sponsel2012improved}.
Typically, when a matrix or tensor lies in the interior
of the copositive cone, the copositivity can be detected
by this type of algorithms. However,
if it lies on the boundary of copositive cone,
these algorithms usually have difficulty in the detection.
For instance, when these methods are applied for
the Hoffman-Pereira matrix in Example~\ref{mat:Hof:Per}
and the Motzkin tensor in Example~\ref{Mot:Rob:Choi},
they cannot detect the copositivity after $10000$ iterations.
However, in the contrast, our Algorithm~\ref{alg:cop}
can detect the copositivity for them in $2$ or $3$ iterations.
An advantage of our method is that
Algorithm~\ref{alg:cop} can determine whether or not
a matrix/tensor is copositive within finitely many iterations,
no matter it lies in the interior, exterior or boundary of the copositive cone.
This property is proved in Theorem~\ref{thm:alg:cvg}.

\bigskip
\noindent
{\bf Acknowledgement}
The research was partially supported by the NSF grants
DMS-1417985 and DMS-1619973.

\bibliographystyle{plain}

\end{document}